\documentclass[11pt]{amsart}

\pdfoutput=1

\usepackage{amsmath}
\usepackage{fullpage}
\usepackage{xspace}
\usepackage[psamsfonts]{amssymb}
\usepackage[latin1]{inputenc}
\usepackage{graphicx,color}
\usepackage[curve]{xypic} 
\usepackage{hyperref}
\usepackage{graphicx}

\usepackage{amsmath}%
\usepackage{amsthm}%
\usepackage{amscd}
\usepackage{amsfonts}%
\usepackage{amssymb}%
\usepackage{graphicx}

\usepackage{mathrsfs}

\usepackage{tikz}
\usetikzlibrary{matrix,arrows}

\usepackage{tikz-cd}

\newtheorem{theorem}{Theorem}[section]

\newtheorem{corollary}[theorem]{Corollary}

\newtheorem{lemma}[theorem]{Lemma}

\theoremstyle{remark}

\numberwithin{equation}{section}

\newcommand{\pfrak}{\mathfrak{p}}

\newcommand{\Acal}{\mathscr{A}}

\newcommand{\Fcal}{\mathscr{F}}

\newcommand{\Lcal}{\mathscr{L}}
\newcommand{\Mcal}{\mathscr{M}}

\newcommand{\Ocal}{\mathscr{O}}

\newcommand{\Xcal}{\mathscr{X}}

\newcommand{\Pro}{\mathbb{P}}

\newcommand{\Z}{\mathbb{Z}}
\newcommand{\C}{\mathbb{C}}

\newcommand{\Q}{\mathbb{Q}}
\newcommand{\R}{\mathbb{R}}

\newcommand{\Aut}{\mathrm{Aut}}

\newcommand{\Spec}{\mathrm{Spec}\,}

\newcommand{\Norm}{\mathrm{N}}

\newcommand{\rk}{\mathrm{rank}\,}

\newcommand{\Ar}{\mathrm{Ar}}
\newcommand{\Fal}{\mathrm{Fal}}
\newcommand{\Cot}{\mathrm{Cot}}

  \DeclareFontFamily{U}{wncy}{}
    \DeclareFontShape{U}{wncy}{m}{n}{<->wncyr10}{}
    \DeclareSymbolFont{mcy}{U}{wncy}{m}{n}
    \DeclareMathSymbol{\Sha}{\mathord}{mcy}{"58}

\makeatletter
\@namedef{subjclassname@2020}{%
  \textup{2020} Mathematics Subject Classification}
\makeatother

\begin{document}
\title[]{Effective Mordell for curves with enough automorphisms}

\author{Natalia Garcia-Fritz}
\address{ Departamento de Matem\'aticas,
Pontificia Universidad Cat\'olica de Chile.
Facultad de Matem\'aticas,
4860 Av.\ Vicu\~na Mackenna,
Macul, RM, Chile}
\email[N. Garcia-Fritz]{natalia.garcia@uc.cl}%

\author{Hector Pasten}
\address{ Departamento de Matem\'aticas,
Pontificia Universidad Cat\'olica de Chile.
Facultad de Matem\'aticas,
4860 Av.\ Vicu\~na Mackenna,
Macul, RM, Chile}
\email[H. Pasten]{hector.pasten@uc.cl}%

\thanks{N.G.-F. was supported by ANID Fondecyt Regular grant 1211004 from Chile. H.P. was supported by ANID Fondecyt Regular grant 1230507 from Chile.}
\date{\today}

\begin{abstract} We prove a completely explicit and effective upper bound for the N\'eron--Tate height of rational points of curves of genus at least $2$ over number fields, provided that they have enough automorphisms with respect to the Mordell--Weil rank of their jacobian. Our arguments build on Arakelov theory for arithmetic surfaces. Our bounds are practical, and we illustrate this by explicitly computing the rational points of a certain genus $2$ curve whose jacobian has Mordell--Weil rank $2$.
\end{abstract}

\maketitle



\section{Introduction} 

\subsection{Height bounds in Mordell's conjecture} 

In 1922, Mordell conjectured that if $X$ is a smooth projective curve over a number field $K$ with genus $g\ge 2$, then the set of $K$-rational points $X(K)$ is finite. This was proved by Faltings in 1983, see \cite{FaltingsMordell}. Another proof was later given by Vojta in 1991 \cite{VojtaMordell} and explained in more classical terms by Bombieri \cite{Bombieri}. More recently, another proof was given by Lawrence and Venkatesh \cite{LawrenceVenkatesh}. However, all these proofs are \emph{ineffective}: they give no height bound for the rational points (or at least an algorithm) that could, in principle, allow their determination. See Section \ref{SecEffectiveHistory} for a discussion of the current attempts to effectivity. 

Our main result (Theorem \ref{ThmMain}) proves the effective Mordell conjecture for curves having sufficiently many automorphisms, with respect to the Mordell--Weil rank of their jacobian. This result gives completely explicit bounds for the height of the rational points of such curves.  As a proof of concept, we use our height bounds to compute all the rational points of the genus $2$ curve
\begin{equation}\label{EqnCurveIntro}
y^2=x^6+x^4+x^2+1
\end{equation}
whose jacobian has Mordell-Weil rank $2$ and has $8$ automorphisms defined over $\Q$; this turns out to be a sufficient number of automorphisms to apply our Theorem \ref{ThmMain}. Other examples are possible; for instance the hyperelliptic curve with LMFDB label 38416.a.614656.1 given by
$$
y^2 = x^6-3x^5-x^4+7x^3-x^2-3x+1
$$ 
has jacobian with Mordell--Weil rank $2$, and it has $12$ automorphisms defined over $\Q$ (see \cite{LMFDB}) which is more than enough in order to apply our results.

Although the specific curve \eqref{EqnCurveIntro} can be analyzed by other methods such as elliptic curve Chabauty \cite{FlynnWetherell},  Manin--Demjanenko \cite{KuleszMateraSchost}, or quadratic Chabauty \cite{BalakrishnanDogra, BianchiPadurariu}, the computations are particularly convenient and we hope they will illustrate how to apply our main results in general.

Before stating our main results in a precise form, we need to introduce some notation.

Let $r\ge 1$ and $n\ge 2$ be integers. For each set $U$ of $n$ different unit vectors in $\R^r$, let $\alpha(U)$ be the smallest angle between two different elements of $U$, and let $\theta(r,n)$ be the supremum of these $\alpha(U)$ as $U$ varies. Given an integer $g\ge 2$ define
$$
\tau(g,r,n)=\begin{cases}
\cos\theta(r,n) -1/g& \mbox{ if } r\ge 2\\
1-1/g& \mbox{ if } r= 1.
\end{cases}
$$
 We will be interested in the cases when $n\ge 3$ and $\tau=\tau(g,r,n)>0$. When $r=1$ and $n\ge 3$ we always have $\tau=1-1/g>0$. On the other hand, here is a table for $r=2$, $g=2$, and the first few values of $\tau$ (truncated to $2$ decimal places):
$$
\begin{array}{c|cccccccccc}
n & 3&4&5&6&7&8&9&10&11&12\\ \hline
\tau & -1 & -0.5& -0.19& 0& 0.12&  0.20& 0.26& 0.30 & 0.34 &0.36
\end{array}
$$
This is because for $r=2$ the angle $\theta(2,n)$ is computed when $U$ consists of the vertices of the regular $n$-gon. (The quantity $\tau$ is closely related to spherical codes.)

In addition to the quantity $\tau(g,r,n)$, we need some Arakelov-theoretical invariants at archimedean and non-archimedean places. 

From now on we fix the following set-up. Let $K$ be a number field, let $X$ a geometrically integral, smooth, projective curve over $K$ of genus $g\ge 2$ whose group of automorphisms defined over $K$ is $G=\Aut_K(X)$. Let $J$ be its jacobian, let $\omega$ be the canonical class of $X$ and consider the map $j:X\to J$ given by the rule $P\mapsto [\omega - 2(g-1)P]$. Let $\hat{h}$ be the N\'eron--Tate height on $J$, normalized to $K$, associated to $2\Theta$ where $\Theta$ is the theta divisor class on $J$.

Given a prime $\pfrak$ in $O_K$, we will define in Section \ref{SecPhip} a quantity $\phi_\pfrak(X)$ which comes from intersection theory on the fibre at $\pfrak$ of the minimal regular model of $X$ over $O_K$. This number $\phi_\pfrak(X)$ is easily computed from such a model, and it vanishes whenever the special fibre at $\pfrak$ has only one component (in particular, for good reduction). Furthermore, given an embedding $v:K\to \C$ we let $X_v$ be the Riemann surface of the $\C$-points of $X$ via $v$, and we write $\delta(Y)$ for the Faltings $\delta$-invariant of a Riemann surface $Y$ of positive genus. 

We put together these invariants by defining the following quantity:
\begin{equation}\label{EqnDefM}
\begin{aligned}
M(X) =\, &  \frac{(g-1)^2}{3}\max\{6,g+1\}\sum_{v:K\to \C}\delta(X_v) \\
&+ 2(g+1)\sum_\pfrak\phi_\pfrak(X)\log \Norm\pfrak\\
&+  2[K:\Q] g (g-1)^2\left(  3g\log g + 16  \right).
\end{aligned}
\end{equation}

With this notation, our main result is:

\begin{theorem}[Main result]\label{ThmMain} Let $H\le G$ be a subgroup of the $K$-rational automorphism group of $X$, let $n=\#H$, and let $r=\rk J(K)$. If $\tau=\tau(g,r,n)>0$, then every $P\in X(K)$ with trivial $H$-stabilizer satisfies
$$
\hat{h}(j(P))\le \frac{M(X)}{2g\tau}.
$$
\end{theorem}

Note that we don't need to compute $J(K)$; a good upper bound for $r$ is enough. There are formulas to approximate the Faltings $\delta$-invariant to high accuracy, see \cite{BostMestreMoretBailly} and more recently \cite{Wilms}. Furthermore, in Theorem \ref{ThmDeltaH} we show how to bound the $\delta$-invariants in terms of the Faltings height of the jacobian of $X$. This allows one to bound the  $\delta$-invariants in practice, as in many cases one can bound $h(J)$, e.g. when $J$ is isogenous to a product of elliptic curves. See Section \ref{SecComputations1} for details.


\subsection{Explicit Mumford gap principle} On the way of proving our main result, we also obtain an explicit version of Mumford's gap principle \cite{Mumford}, which can be of independent interest. For this we denote by $\langle-,-\rangle$ the N\'eron--Tate pairing on $J(K)$, namely
$$
2\langle x, y\rangle :=\hat{h}(x+y)-\hat{h}(x)-\hat{h}(y). 
$$
Also, we recall that $M(X)$ was defined in \eqref{EqnDefM}.
\begin{theorem}[Explicit gap principle] \label{ThmGap} For all pairs of different points $P,Q\in X(K)$ we have
$$
\hat{h}(j(P))+\hat{h}(j(Q))- 2g\langle j(P),j(Q)\rangle \ge -M(X).
$$
In particular, if $\hat{h}(j(P))$ and $\hat{h}(j(Q))$ are non-zero, and if $\theta$ is the angle between $j(P)$ and $j(Q)$ in $J(K)\otimes \R$ using the N\'eron--Tate pairing, then
$$
\cos \theta \le \frac{M(X)}{2g\sqrt{\hat{h}(j(P))\hat{h}(j(Q))}}+ \frac{1}{2g}\left(\sqrt{\frac{\hat{h}(j(P))}{\hat{h}(j(Q))}}+\sqrt{\frac{\hat{h}(j(Q))}{\hat{h}(j(P))}}\right).
$$ 
\end{theorem}

This inequality provides gaps between rational points of $X$: when two such points have large height of similar size, their angle $\theta$ is bounded from below (i.e., $\cos\theta$ is bounded away from $1$).

The proof is an Arakelov-theoretical computation similar to that in \cite{SzpiroPeu}, although for our purposes we need to make everything explicit and we don't assume semistability. In the semistable case  our proof can be seen as bounding the ``Mumford constant'' introduced by Szpiro \cite{SzpiroPeu}.


\subsection{Other approaches to effective Mordell}\label{SecEffectiveHistory}

In the direction of effective Mordell there are three main approaches: The Chabauty--Coleman--Kim method \cite{Chabauty, Coleman, Kim}, the Manin--Demjanenko method \cite{Manin, Demjanenko} and modularity estimates \cite{MurtyPasten, Alpoge}. 

The Chabauty--Coleman method requires that the Mordell--Weil rank $r$ of the jacobian $J$ of $X$ be less than the genus $g$. Under this assumption it gives good bounds for the \emph{number} of rational points. The method of proof has been adapted to find, in many cases, the rational points of certain curves, see \cite{McCallumPoonen} for an introduction. See also \cite{BBBLMTVrecentpadic} for an exposition on non-abelian extensions originated in the work of Kim \cite{Kim} that, in some cases, can relax the condition $r<g$ leading to spectacular applications such as \cite{CursedCurve}. However, we remark that at present it is not known whether the Chabauty--Coleman--Kim method terminates when its hypothesis are satisfied. This is different from our results; when our hypotheses are satisfied, the search for rational points using our height bounds is guaranteed to terminate.

The Manin--Demjanenko method requires that $X$ admits several linearly independent maps to an abelian variety whose rank is less than the number of maps. It is effective in theory, but it is quite difficult to apply in practice, as the standard proof is not very explicit \cite{SerreMW}. Nevertheless, some examples have been successfully analyzed, see for instance \cite{GirardKulesz, Kulesz, KuleszMateraSchost}. Let us mention that a variation of the Manin--Demjanenko method has been studied in detail when $X$ is contained in a power of an elliptic curve of small rank, see for instance \cite{CVVappendixStoll} and the references therein. 

Finally, the third method (modularity) attempts to associate to each rational point of $X$ a modular abelian variety whose Faltings height can be bounded using the theory of automorphic forms. This was first used in the context of the unit equation by Murty--Pasten \cite{MurtyPasten} giving explicit practical bounds, then it was applied to other Diophantine equations involving $S$-integral points \cite{vonKanelMatschke} and most recently, it was used for rational points in some special curves \cite{Alpoge} although the estimates in this last work are not practical at present. 

Of these three methods, our Theorem \ref{ThmMain} is closer to the Manin--Demjanenko method although the proof is completely different, our bounds are explicit, and it does not seem (to the authors) that either contains the other.

In conditional grounds, it is known that an effective version of the $ABC$ conjecture implies the effective Mordell conjecture  in the form of height bounds for rational points \cite{Elkies}. Similarly, arithmetic canonical class inequalities with effective constants imply effective Mordell in the form of height bounds for rational points, see for instance \cite{SzpiroNumeriques, ParshinBMY} and Vojta's appendix in \cite{Lang}. In another direction, if one assumes the Hodge, Tate, and Fontaine--Mazur conjectures then there is an algorithm to compute $X(K)$, see \cite{AlpogeLawrence}; at present this conditional algorithm is far from practical.



\section{Metrics}

\subsection{Admissible metrics}\label{SecAdmissible}

Let $Y$ be a compact Riemann surface of genus $g\ge 1$. On $H^0(Y,\Omega^1_{Y/\C})$ we put the inner product
$$
( \alpha,\beta)  =\frac{i}{2}\int_Y \alpha\wedge \overline{\beta}.
$$

Let $\alpha_1,...,\alpha_g$ be an orthonormal basis for $H^0(Y,\Omega^1_{Y/\C})$. We define the $(1,1)$-form
$$
d\mu^{\Ar} = \frac{i}{2g}\sum_{j=1}^g \alpha_j\wedge\overline{\alpha_j}.
$$
This $(1,1)$-form is independent of the choice of orthonormal basis $\{\alpha_j\}_j$ and it defines a probability measure on $Y$. It is called the \emph{Arakelov $(1,1)$-form} on $Y$.

We follow the same setting developed in Section 3 of \cite{FaltingsCalculus}. In particular, there is the notion of admissible metric $\|-\|$ on a line sheaf $\Lcal$ on $Y$ (namely, its curvature is a scalar multiple of $d\mu^{\Ar}$), which always exists and it is unique up to a scalar factor. If $\Lcal=\Ocal(D)$ for a divisor $D$, then this metric is in fact unique under the normalization
$$
\int_Y \log \|1\| d\mu^{\Ar} =0
$$
where $1$ is the canonical meromorphic section of $\Ocal(D)$.

\subsection{Green functions} The Green function $G(P,Q)$ associated to the canonical $(1,1)$-form $\mu^{\Ar}$ is a function defined on $Y\times Y$ which is real valued symmetric and continuous, vanishing only along the diagonal to order $1$ and with a curvature normalization relative to $\mu^{\Ar}$. Its precise definition can be found in Section 3 of \cite{FaltingsCalculus}. The Green function can be thought as defining a distance between points $P,Q\in Y$. Furthermore, it gives a metric $\|-\|_P$ on $\Ocal(P)$ by defining 
$$
\|1\|_P(Q)=G(P,Q).
$$
This is the unique admissible normalized metric for $\Ocal(P)$ with respect to $\mu^{\Ar}$, in the sense of Section \ref{SecAdmissible}. 

In order to study the Arakelov intersection pairing later, for $P\ne Q$ one defines the \emph{canonical logarithmic Green function}
$$
\gamma(P,Q) = \log G(P,Q).
$$

\subsection{The canonical metric}\label{SecCanMetric} Consider the sheaf $\Omega^1_{Y/\C}$. For every point $P\in Y$ there is the residue isomorphism
$$
\Omega^1_{Y/\C}|_P\otimes \Ocal(P)|_P\to \C
$$
defined by the rule $dt\otimes t^{-1}\mapsto 1$, for any local parameter $t$ at $P$. Endowing $\C$ with the usual absolute value and $\Ocal(P)$ with its unique admissible normalized metric in the sense of Section \ref{SecAdmissible}, we obtain a metric on each fibre of $\Omega^1_{Y/\C}$. Using Green functions one proves the following (see \cite{FaltingsCalculus}):

\begin{theorem} The metrics on $\Omega^1_{Y/\C}|_P$ of the previous construction vary smoothly with $P$ and define a metric on $\Omega^1_{Y/\C}$. This metric is admissible.
\end{theorem}

The metric on $\Omega^1_{Y/\C}$ thus obtained will be called the \emph{canonical metric}.

\subsection{Faltings $\delta$-invariant} Faltings introduced in p.402 of \cite{FaltingsCalculus} an important  numerical invariant $\delta(Y)$ of a Riemann surface $Y$ of genus $g\ge 1$, now called the \emph{Faltings $\delta$-invariant}. 

Let $\Mcal_g(\C)$ be the moduli space of genus $g$ compact Riemann surfaces and let $p_Y\in \Mcal_g(\C)$ be the point corresponding to $Y$. Intuitively, $\delta(Y)$ can be thought as $-\log$ of the distance of $p_Y$ to the boundary of $\Mcal_g(\C)$ in a suitable compactification. 

The invariant $\delta(Y)$ can very well take negative values. An explicit lower bound is provided by Wilms in Corollary 1.2 of \cite{Wilms}:

\begin{theorem}\label{ThmWilms} Let $Y$ be a compact Riemann surface of genus $g\ge 1$. Then
$$
\delta(Y)\ge -2g\log (2\pi^4).
$$
\end{theorem}

The $\delta$-invariant can be used to bound Green functions, see Corollary 1.5 in \cite{Wilms}:

\begin{theorem}\label{ThmBdGreen} Let $Y$ be a compact Riemann surface of genus $g\ge 1$ and let $\gamma(P,Q)$ be its canonical logarithmic Green function. Then for all $P\ne Q$ in $Y$ we have
$$
\gamma (P,Q) < \frac{1}{24g}\max\{6,g+1\}\delta(Y) + \frac{3g}{4}\log g +4.
$$
\end{theorem}

\subsection{Norms for differential forms on abelian varieties} \label{SecNormAV}

Let $A$ be a complex abelian variety of dimension $g$. On $H^0(A,\Omega^g_{A/\C})$ we put the inner product
$$
(\alpha,\beta)=\left(\frac{i}{2}\right)^g\int_A \alpha\wedge \overline{\beta}
$$
and let $\|-\|_{A,\Fal}$ be its associated norm. We have a canonical isomorphism of $1$-dimensional vector spaces
$$
\det H^0(A,\Omega^1_{A/\C}) = H^0(A,\Omega^g_{A/\C}).
$$ 
This induces a norm on $\det H^0(A,\Omega^1_{A/\C}) $ which we also denote by  $\|-\|_{A,\Fal}$. 

The norm we just constructed will be used to define the Faltings height of an abelian variety over a number field.


\section{Review of Arakelov theory}

\subsection{Arithmetic surfaces} From now on, $K$ is a number field, $R=O_K$, and $B=\Spec R$.

An \emph{arithmetic surface} is a pair $(\Xcal,\pi)$ where $\Xcal$ is an integral scheme, $\pi:\Xcal\to B$ is a morphism of finite type, flat, and proper, and the generic fibre $\Xcal_\eta$ is a smooth and geometrically irreducible curve over $K$. Since $\pi$ is flat and proper, the required properties on $\Xcal_\eta$ imply that the fibres of $\pi$ are connected.

We say that $(\Xcal,\pi)$ is \emph{regular} if $\Xcal$ is a regular scheme. It is a theorem of Lichtenbaum (Theorem 2.8 in \cite{Lichtenbaum}) that regular arithmetic surfaces are projective over $B$.

\subsection{Minimal regular models and good reduction} Let $X$ be a geometrically irreducible, smooth, projective curve over $K$ of genus $g$.

An \emph{integral model} for $X$ is a triple $(\Xcal,\pi, i)$ where $(\Xcal,\pi)$ is an arithmetic surface and $i:X\to \Xcal_\eta$ is an isomorphism.

Integral models exist: we might embed $X$ in some $\Pro^n_K$ and then take the closure of $X$ in $\Pro^n_B$. 

We say that the integral model is \emph{regular} if $\Xcal$ is regular. Furthermore, regular integral models exist, which can be seen by desingularization of any integral model; see Corollary 8.3.51 in \cite{Liu}.

Given integral models $(\Xcal,\pi, i)$ and $(\Xcal',\pi', i')$, a morphism of integral models is a morphism $f:\Xcal\to \Xcal'$ which is compatible with $\pi$, $\pi'$, $i$, and $i'$ in the sense that it is a $B$-morphism satisfying $i'=f\circ i$.

We say that a regular integral model $(\Xcal,\pi, i)$ is \emph{minimal} if for every regular integral model $(\Xcal',\pi', i')$ and every proper birational morphism of integral models $f:\Xcal\to \Xcal'$ we have that $f$ is an isomorphism.  When $g\ge 1$, there is a unique minimal regular model up to isomorphism, see Proposition 10.1.8 in \cite{Liu} (note that the terminology in \emph{loc. cit.} is slightly different).

We say that $X$ has \emph{good reduction} at a closed point $\pfrak\in B$ if for the minimal regular model $(\Xcal,\pi, i)$ of $X$ the fibre $\Xcal_\pfrak$ is a smooth curve over the residue field $\kappa(\pfrak)=R/\pfrak$. Otherwise, we say that $X$ has \emph{bad reduction} at $\pfrak$.

Let $(\Xcal',\pi',i')$ be any integral model for $X$ and $\pfrak\in B$ is a closed point. We note that if $\Xcal'_\pfrak$ is smooth, then $\pfrak$ is a prime of good reduction (of course the converse does not hold in general).

\subsection{Degree} Let $M$ be a projective module of rank $1$ over $R$ and for each $v:K\to \C$ let $\|-\|_v$ be a hermitian norm on $M\otimes_v\C$. The degree of $\overline{M}=(M,\{\|-\|_v\}_v)$ is defined as
$$
\deg_R \overline{M}=\log \# (M/Rm)-\sum_v \log \|m\|_v
$$ 
for any non-zero $m\in M$; this is well-defined by the product formula. Note that the sum is over embeddings $v:K\to \C$, not over the archimedean places (thus, a non-real place corresponds to two different embeddings), and that  we are not dividing by $[K:\Q]$.

\subsection{Intersections}  Let $(\Xcal,\pi)$ be a regular arithmetic surface, assume that $X=\Xcal_\eta$ has genus $g\ge 1$,  and for each embedding $v:K\to \C$ let $X_v$ be the compact Riemann surface associated to $X$ via the embedding $v$.

A \emph{metrized line sheaf} is a pair $\overline{\Lcal}=(\Lcal,\{\|-\|_v\}_v)$ where $\Lcal$ is a line sheaf on $\Xcal$ and for each $v:K\to \C$ we have that $\|-\|_v$ is an admissible metric on $\Lcal_v=\Lcal|_{X_v}$.

On the other hand, an \emph{Arakelov divisor} is an expression of the form
$$
D= D_{\rm fin} +\sum_{v:K\to \C} \alpha_v\cdot F_v
$$
where $D_{\rm fin}$ is a divisor on $\Xcal$, the symbol $F_v$ can be thought as the ``archimedean fibre at $v$'', and $\alpha_v$ are real numbers.

Associated to an Arakelov divisor $D$ there is the metrized line sheaf $\overline{\Ocal}(D)$: on the finite part we take the sheaf $\Ocal(D_{\rm fin})$, and for the metrics we multiply the admissible metric of $\Ocal(D_{\rm fin}|_{X_v})$ by $\exp(-\alpha_v)$.

Arakelov \cite{Arakelov, ArakelovICM} defined an intersection pairing between isometric isomorphism classes of metrized line sheaves. We denote this pairing by $(\overline{\Lcal}.\overline{\Fcal})$. It is $\R$-valued and bilinear for tensor products. It induces a bilinear pairing on Arakelov divisors by
$$
(D.E):=(\overline{\Ocal}(D).\overline{\Ocal}(E)).
$$
It also makes sense to intersect an Arakelov divisor with a metrized line sheaf.

\subsection{Special cases}\label{SecSpecial}
 
A couple of special cases of the intersection pairing deserve attention for our purposes. For a $K$-rational point  $P\in X(K)$ we let $D_P$ be the closure of $P$ in $\Xcal$. It is also an Arakelov divisor, where the archimedean part is $0$.

\begin{itemize}

\item If $P\in X(K)$ and $s:B\to \Xcal$ is the corresponding section of $\pi$, then
$$
(D_P.\overline{\Lcal})=\deg_R s^*\overline{\Lcal}
$$
where $s^*\overline{\Lcal}$ is the projective module $H^0(B,s^*\Lcal)$ endowed with the fiber metrics on each $\Lcal_v|_{v(P)}$, and $v(P)$ is the image of $P$ via the map $X(K)\to X_v$ induced by the embedding $v:K\to \C$.  

\item If $P,Q\in X(K)$ are different, then
$$
(D_P.D_Q)= (\mbox{geometric intersection contribution}) + \sum_{v:K\to \C} -\gamma_v(v(P),v(Q))
$$
where $\gamma_v$ is the canonical logarithmic Green function on $X_v$, and the first contribution is non-negative. In particular
$$
(D_P.D_Q)\ge - \sum_{v:K\to \C}  \gamma_v(v(P),v(Q)).
$$

\item If $F$ is a fibral divisor on $\Xcal$ above a finite place $\pfrak\in B$  and we intersect it with an Arakelov divisor or a metrized line sheaf, we obtain the usual (geometric) intersection number multiplied by $\log \Norm \pfrak$. 

\end{itemize}

We take this opportunity to introduce some notation: the geometric intersection on $\Xcal$ for divisors not sharing components, will be denoted by $[D.E]$ (this is a sum of intersection multiplicities). It does not agree with the Arakelov intersection pairing.

\subsection{The canonical sheaf} 

As $(\Xcal,\pi)$ is a regular arithmetic surface, we have that $\pi:\Xcal\to B$ is a local complete intersection (cf. Example 6.3.18 in \cite{Liu}) so, the fibres of $\pi$ are Gorenstein and the relative dualizing sheaf of $\pi$ exists and it is invertible, according to Grothendieck's duality theory (cf. \cite{DeligneRapoport}). We denote the relative dualizing sheaf of $\pi$ by $\omega$. 

Let $X=\Xcal_\eta$ be of genus $g\ge 1$ as before. For each $v:K\to \C$ we have the canonical metric on $\omega|_{X_v}=\Omega^1_{X_v/\C}$ from Section \ref{SecCanMetric}. The sheaf $\omega$ endowed with these canonical metrics is called the \emph{Arakelov canonical sheaf}, and it is denoted by $\widehat{\omega}$.

\subsection{The case of regular minimal arithmetic surfaces} In the context of Arakelov theory for \emph{semistable} arithmetic surfaces, Faltings \cite{FaltingsCalculus} proved arithmetic analogues of the Riemann-Roch theorem, Noether's formula, the Hodge index theorem (see also \cite{Hriljac}), and the adjunction formula, as well as the fact that $(\widehat{\omega}.\widehat{\omega})\ge 0$; note that Noether's formula in the semistable case was later proved with explicit constants by Moret-Bailly \cite{MoretBailly}.

 Lang \cite{Lang} extended part of this theory to regular arithmetic surfaces, including the Riemann-Roch theorem, the Hodge index theorem (after Hriljac \cite{Hriljac}, see Section \ref{SecHodge}), and the adjunction formula, see also \cite{Chinburg}. The adjunction formula for sections is:

\begin{theorem}\label{ThmAdjunction} Let $(\Xcal,\pi)$ be a minimal regular arithmetic surface whose generic fibre $X=\Xcal_\eta$ has genus $g\ge 1$. For all $K$-rational points $P\in X(K)$ we have 
$$
(D_P.D_P) + (D_P.\widehat{\omega})= 0.
$$
\end{theorem}

 A version of Noether's formula for regular arithmetic surfaces was proved by Saito \cite{Saito}, and $(\widehat{\omega}.\widehat{\omega})\ge 0$ was obtained by Sun \cite{Sun}.  
\begin{theorem}\label{ThmSelf} Let $(\Xcal,\pi)$ be a minimal regular arithmetic surface whose generic fibre $X=\Xcal_\eta$ has genus $g\ge 2$. We have $(\widehat{\omega}.\widehat{\omega})\ge 0$.
\end{theorem}

\subsection{The Faltings height of an abelian variety}

Let $A$ be an abelian variety of dimension $g\ge 1$ over $K$ and let $\Acal$ be its N\'eron model over $R$. Let $\epsilon:B\to \Acal$ be the identity section. Define the cotangent space at the identity as
$$
\Cot(\Acal)=H^0(B,\epsilon^*\Omega^1_{\Acal/B}).
$$
Then $\Cot(\Acal)$ is a projective $R$-module of rank $g$ and it follows that 
$$
M=\det \Cot(\Acal)
$$
is a projective $R$-module of rank $1$. For each $v:K\to \C$ we have that
$$
M\otimes_v \C = \det H^0(A_v , \Omega^1_{A_v/\C}) =  H^0(A_v , \Omega^g_{A_v/\C})
$$
where $A_v=A\otimes_v \C$. Thus, each $M\otimes_v \C$ has the Hermitian norm $\|-\|_{A_v, \Fal}$ from Section \ref{SecNormAV} and we obtain the metrized projective $R$-module of rank $1$ 
$$
\overline{\det \Cot(\Acal)} = (\det \Cot(\Acal), \{\|-\|_{A_v, \Fal}\}_v).
$$
The \emph{Faltings height} of $A$ relative to $K$ is
$$
h_K(A)=\deg_R \overline{\det \Cot(\Acal)} .
$$
(Note that the subindex $K$ is intended to remind the reader that we are not normalizing to $\Q$.) This is a measure of complexity of $A$ introduced by Faltings in \cite{FaltingsMordell}. There he proved the following comparison under isogeny:

\begin{theorem} Let $\phi:A_1\to A_2$ be an isogeny of abelian varieties over $K$ and assume that they are principally polarized. Then
$$
h_K(A_2) -\frac{[K:\Q]}{2}\log \deg \phi\le  h_K(A_1)\le h_K(A_2)  + \frac{[K:\Q]}{2}\log \deg \phi.
$$
\end{theorem}
\begin{proof} The lower bound is Lemma 5 in \cite{FaltingsMordell} while the upper bound is the same using the dual isogeny; we have $A_j^\vee\simeq A_j$ for the dual abelian varieties due to the principal polarizations.
\end{proof}

Our main interest on the Faltings height is that it can be used to bound the $\delta$-invariant.

\begin{theorem}\label{ThmDeltaH} Let $X$ be a geometrically irreducible, smooth, projective curve over $K$ of genus $g\ge 1$. Let $J$ be the jacobian of $X$. Then
$$
\sum_{v:K\to \C} \delta(X_v) \le 12 h_K(J)+ 4g[K:\Q]\log(2\pi).
$$
\end{theorem}
\begin{proof} If we base change to a finite extension $L/K$ we have $h_L(J_L)\le [L:K] h_K(J)$, while the right hand side of the bound we want to prove gets multiplied by $[L:K]$. Therefore, we may assume that $A$ has semistable reduction over $K$.

In \cite{FaltingsCalculus} Faltings proved an arithmetic version of Noether's formula with some unspecified numerical constants. This was made explicit by Moret-Bailly \cite{MoretBailly}:
$$
12\deg_R (\det \pi_*\widehat{\omega} ) = (\widehat{\omega}.\widehat{\omega}) + \sum_{\pfrak\in B} \delta_\pfrak(\Xcal)\log \Norm \pfrak + \sum_{v:K\to \C} \delta(X_v) - 4g[K:\Q]\log (2\pi)
$$
where $\delta_\pfrak(\Xcal)$ is a non-negative contribution measuring singularities of the special fibre $\Xcal_\pfrak$ (the metrics used by Moret-Bailly agree with ours and with those of Faltings \cite{FaltingsMordell, FaltingsCalculus}). Also, we recall that $(\widehat{\omega}.\widehat{\omega})\ge 0$. 

The equality
$$
\deg_R (\det \pi_*\widehat{\omega} ) = h_K(J)
$$
can be found in p.71 of  \cite{BostMestreMoretBailly}, see also p.159 in \cite{Lang} or Section 4.4 in \cite{EdixhovenDeJong}. This proves the desired estimate.
\end{proof}


\section{Fibres}

\subsection{Fibral corrections}\label{SecFibral1}  Let $(\Xcal,\pi)$ be a minimal regular arithmetic surface whose generic fibre $X=\Xcal_\eta$ has genus $g\ge 2$.

Let $W^0$ be the group of Arakelov divisors whose restriction to the generic fibre $X$ has degree $0$.  For $D\in W^0$ we would like to consider a ``correction term'' $\Phi(D)$ such that $D-\Phi(D)$ has intersection $0$ with every fibre component at every fibre of $\pi$. This is indeed possible as long as we allow divisors with $\Q$-coefficients (which of course is compatible with the Arakelov intersection pairing, by linearity).

For a prime $\pfrak \in B$ we let $V_\pfrak^0$ be the $\Q$-vector space generated by the reduced irreducible components of the fibre $\Xcal_\pfrak$, having coefficients that add up to $0$. See \cite{Lang} for the following result.

\begin{theorem} There is a linear map $\Phi_\pfrak: W^0\to V^0_\pfrak$ such that for all $D\in W^0$ we have that $D-\Phi_\pfrak(D)$ is orthogonal to $V^0_\pfrak$ for the Arakelov intersection pairing. In particular, if we define $\Phi(D)=\sum_\pfrak \Phi_\pfrak(D)$ then $D-\Phi(D)$ is orthogonal to all vertical Arakelov divisors.

Furthermore, if $\Xcal_\pfrak$ is irreducible, then $\Phi_\pfrak$ is the zero map.
\end{theorem}

\subsection{Computation of the fibral corrections}\label{SecFibral2} In practice, for $D\in W^0$ the divisor $\Phi_\pfrak(D)$ is easy to find if the special fibre $\Xcal_\pfrak$ is known; this immediately follows from the standard proof of the previous theorem, see for instance \cite{Lang}. Let us explain the practical procedure. For this we fix $\pfrak$.

Let $C_1,...,C_s$ be the reduced irreducible components of $\Xcal_\pfrak$ and let $m_1,...,m_s$ be their multiplicities. We write $F_\pfrak = \sum_jm_jC_j$; note that this is a divisor, not the special fibre $\Xcal_\pfrak$ (which is a scheme).

The next observation is due to the fact that for  the Arakelov intersection pairing, all fibres are numerically equivalent.

\begin{lemma}\label{LemmaW0Fp} If $D\in W^0$ then the intersection of $D$ with every $F_\pfrak$ is $0$.
\end{lemma}

The next lemma is a standard fact about fibred surfaces, see for instance the proof in \cite{FaltingsCalculus}.

\begin{lemma}\label{LemmaSemiDef} The geometric intersection pairing $[-.-]$ is negative semidefinite on the group of divisors generated by the $C_j$, and the only divisors in this group with $0$ self-intersection are multiples of $F_\pfrak=\sum_j m_jC_j$. (The same holds for the Arakelov intersection pairing, after multiplying by $\log\Norm\pfrak$.) 
\end{lemma}

Let $\Mcal = ([C_i.C_j])_{i,j}$ be the intersection matrix of the divisors $C_j$ in the given order. This is a symmetric matrix with integer coefficients and it can be regarded as a linear map $\Mcal:\Q^s\to \Q^s$. Let $U,H\subseteq \Q^s$ be the hyperplanes defined by the equations $\sum_j x_j=0$ and $\sum_i m_ix_i=0$ respectively.

\begin{lemma}\label{LemmaUH} $\Mcal$ restricts to a linear bijection $U\to H$.
\end{lemma}
\begin{proof}
Let ${\bf a}=(a_j)_j\in \Q^s$, then
$$
\Mcal {\bf a} = \left(\sum_ja_j [C_i.C_j]\right)_i
$$
and using Lemma \ref{LemmaW0Fp} we see that this vector is in $H$:
$$
\sum_im_i\sum_ja_j [C_i.C_j] = \left[ F_\pfrak . \sum_ja_jC_j \right] = 0.
$$
We have ${\bf a}\in \ker \Mcal$ if and only if for each $i$ we have
$$
0=\sum_ja_j [C_i.C_j]=\left[ C_i . \sum_ja_jC_j \right].
$$
This implies
$$
\left[ \sum_ja_jC_j . \sum_ja_jC_j \right]=0
$$
which gives $\sum_ja_jC_j=\lambda F_\pfrak$ for some $\lambda$, by Lemma \ref{LemmaSemiDef}. 

Let us now assume ${\bf a}\in U$. Since the coefficients of $F_\pfrak$ are $m_j>0$ but $\sum a_j=0$, we deduce $\lambda=0$. This means that $\Mcal$ is injective on $U$. We conclude because $\dim U=\dim H=s-1$.
\end{proof}

\begin{corollary}\label{CoroPhiFind} For $D\in W^0$, the divisor
$$
\Phi_\pfrak(D)=\sum_j a_j C_j
$$
is found by solving the system of $r+1$ linear equations in $r$ variables
$$
\begin{cases}
\sum_{j}a_j[C_1.C_j] &=  [C_1.D_{\rm fin}]\\
&\vdots \\
\sum_{j}a_j[C_s.C_j] &=  [C_s.D_{\rm fin}]\\
\sum_j a_j &=0
\end{cases}
$$
which has a unique rational solution.
\end{corollary}
\begin{proof} Note that these equations are the defining conditions of the desired divisor $\Phi_\pfrak(D)\in V_\pfrak^0$. The solution exists and it is unique because 
$$
\sum_i m_i [C_i.D_{\rm fin}] = [F_\pfrak. D_{\rm fin}] = 0
$$
by definition of $W^0$. Here we use $(E.D)=[E.D_{\rm fin}]\log \Norm\pfrak$ when $E$ is supported on the fibre at $\pfrak$. Then we apply Lemma \ref{LemmaUH}.
\end{proof}

We note that if $\Xcal_\pfrak$ is irreducible (for instance, if $\pfrak$ is a prime of good reduction) then in fact $\Phi_\pfrak(D)=0$.

\subsection{The quantities $\phi_\pfrak(X)$}\label{SecPhip}

In this section we define the numbers $\phi_\pfrak(X)$.  We keep the notation of Sections \ref{SecFibral1} and \ref{SecFibral2}. Let $g_j$ be the arithmetic genus of $C_j$ (recall that $C_j$ is reduced).  The next result is a consequence of Proposition 9.1.35 and theorem 9.1.37 from \cite{Liu}.
\begin{lemma}\label{LemmaSum0} We have 
$$
[F_\pfrak.\omega] = 2(g-1)  =  \mu_\pfrak+\sum_j 2(g_j-1)m_j
$$
where $\mu_\pfrak = -\sum_{j} m_j[C_j.C_j]\ge 0$. Furthermore, $\mu_\pfrak=0$ when $\Xcal_\pfrak$ is irreducible.
\end{lemma}

\begin{lemma} \label{LemmaOrth} Let $P\in X(K)$. In the previous setting, the divisor 
$$
\omega - 2(g-1)D_P
$$
is orthogonal to all the divisors $F_\pfrak$.
\end{lemma}
\begin{proof} Note that $D_P$ intersects $F_\pfrak$ at a single component, of multiplicity $1$. Thus we get
$$
[F_\pfrak\, .\, \omega - 2(g-1)D_P ] = [F_\pfrak.\omega]- 2(g-1)=0
$$
by Lemma \ref{LemmaSum0}.
\end{proof}

Given a prime $\pfrak$ we let $J_\pfrak$ be the set of indices $j$ for which $m_j=1$. Let $d_{i,k}$ be $1$ if $i=k$ and $0$ otherwise. For each $k\in J_\pfrak$ let  consider the linear system 
$$
\Xi_k=\Xi_k(\pfrak):\quad 
\begin{cases}
\sum_{j}b^{(k)}_j[C_1.C_j] &=  2(g_1-1) -\frac{2}{m_1}(g-1)d_{1,k} -  [C_1.C_1] \\
&\vdots \\
\sum_{j}b^{(k)}_j[C_s.C_j] &=  2(g_r-1) -\frac{2}{m_s}(g-1)d_{s,k} - [C_s.C_s]\\
\sum_j b^{(k)}_j &=0
\end{cases}
$$
on the variables $b^{(k)}_1,...,b^{(k)}_s$.

\begin{lemma} For each $k\in J_\pfrak$, the system $\Xi_k$ has a unique rational solution. 
\end{lemma}
\begin{proof}
By Lemma \ref{LemmaSum0} we see that the sum of the expressions on the right hand side of the previous equations, weighted by the $m_j$'s, is
$$
\sum_j 2(g_j-1)m_j - 2(g-1) - \sum_j m_j[C_j.C_j] =\sum_j 2(g_j-1)m_j - 2(g-1) +\mu_{\pfrak} =0.
$$
We conclude by  Lemma \ref{LemmaUH}.
\end{proof}

Finally, the quantity $\phi_\pfrak(X)$ is defined as the maximum, as $k$ varies in $J_\pfrak$, of the absolute value of the (geometric) self-intersection
$$
\left[ \sum_j  b^{(k)}_j C_j\, .\,\sum_j  b^{(k)}_j C_j\right].
$$
We note that $\phi_\pfrak(X)\ge 0$, it equals $0$ if $\Xcal_\pfrak$ is irreducible, and it is rational. 

\begin{theorem}\label{Thmphip} In the previous setting, for every $P\in X(K)$ and every prime $\pfrak$ we have
$$
\left|[\Phi_\pfrak(\widehat{\omega} - 2(g-1)D_P)\, .\, \Phi_\pfrak(\widehat{\omega} - 2(g-1)D_P )]\right| \le \phi_\pfrak.
$$
\end{theorem}
\begin{proof}
Let $k$ be such that $D_P$ meets $C_k$; this index is unique and in fact $k\in J_\pfrak$. By Corollary \ref{CoroPhiFind}, to compute $\Phi_\pfrak(\widehat{\omega} - 2(g-1)D_P)$ we must solve the linear system
$$
\begin{cases}
\sum_{j}a_j[C_1.C_j] &=  [C_1\, .\, \omega - 2(g-1)D_P]\\
&\vdots \\
\sum_{j}a_j[C_s.C_j] &=  [C_s\, .\, \omega - 2(g-1)D_P]\\
\sum_j a_j &=0.
\end{cases}
$$
We note that by adjunction
$$
[C_j\, .\, \omega - 2(g-1)D_P] = [C_j.\omega] - 2(g-1)d_{j,k} = 2(g_j-1)-[C_j.C_j]- 2(g-1)d_{j,k} 
$$
so the previous linear system is precisely $\Xi_k=\Xi_k(\pfrak)$. 
\end{proof}

\subsection{The Faltings--Hriljac theorem}\label{SecHodge} We keep the notation of Sections \ref{SecFibral1} and \ref{SecFibral2}.

Let $J$ be the Jacobian of $X$ and let $\Theta$ be the Theta divisor class on $J$. We let $\hat{h}$ be the N\'eron--Tate height on $J(\overline{K})$ associated to the divisor class $2\Theta$, normalized to $K$.

The rule $P\mapsto [\Omega^1_X\otimes \Ocal(-2(g-1) P )]$ defines a map $j:X\to J$ which, in general is not an embedding. In this way $\hat{h}\circ j$ is a canonically defined height on $X$.

Faltings \cite{FaltingsCalculus} (in the semistable case) and Hriljac \cite{Hriljac} (in general) proved an important relation between $\hat{h}$ and the Arakelov intersection pairing.

\begin{theorem} Let $D$ be an Arakelov divisor orthogonal to all fibre components  (in particular, the degree of $D|_X$ on $X$ is $0$). Then $D|_X$ defines an element $[D]\in J(K)$ which staisfies
$$
(D.D)=-\hat{h}([D]).
$$ 
\end{theorem}

To ease notation, the Arakelov self-intersection will be denoted by a square.  For a $K$-rational point $P$ let us define $\Psi_{\pfrak,P}=\Phi_\pfrak(\widehat{\omega}-2(g-1)D_P)$ and $\Psi_{P}=\sum_{\pfrak}\Psi_{\pfrak,P}$. We note that if $\Xcal_\pfrak$ is irreducible then $\Psi_{\pfrak,P}=0$.

From Lemma \ref{LemmaOrth} we deduce:

\begin{corollary}\label{CoroHhat} For all $K$-rational points $P\in X(K)$ we have
$$
(\widehat{\omega} - 2(g-1)D_P - \Psi_P)^2= -\hat{h}(j(P)).
$$
\end{corollary}


\section{Proof of the results}

\subsection{Effective Mumford gap} 

\begin{proof}[Proof of Theorem \ref{ThmGap}] We will use Lemma \ref{LemmaOrth} and Corollary \ref{CoroHhat}. For short let us write 
$$
E_P=\widehat{\omega} - 2(g-1)D_P
$$
and similarly for $E_Q$. To start, let us first drop the assumption $P\ne Q$.

Since $\Psi_P$ and $\Psi_Q$ are fibral Arakelov divisors, we have
$$
\begin{aligned}
-\langle j(P),j(Q)\rangle &= -\frac{1}{2}\left(\hat{h}(j(P)+j(Q))-\hat{h}(j(P))-\hat{h}(j(P))\right)\\
&= (E_P-\Psi_P\, .\, E_Q-\Psi_Q)\\
&= (E_P.E_Q) - (\Psi_P.E_Q)\\
&=(E_P.E_Q) - (\Psi_P\, .\, E_Q-\Psi_Q+ \Psi_Q)\\
&=(E_P.E_Q)- (\Psi_P.\Psi_Q).
\end{aligned}
$$
Next we have
$$
\begin{aligned}
(E_P.E_Q)&= \widehat{\omega}^2 - 2(g-1)((D_P+D_Q).\widehat{\omega}) + 4(g-1)^2(D_P.D_Q)
\end{aligned}
$$
from which we find
$$
\langle j(P),j(Q)\rangle = 2(g-1)((D_P+D_Q).\widehat{\omega}) - 4(g-1)^2(D_P.D_Q) -\widehat{\omega}^2+ (\Psi_P.\Psi_Q).
$$
In particular, taking $P=Q$ and using arithmetic adjunction (Theorem \ref{ThmAdjunction}) one gets
$$
\hat{h}(j(P))=4g(g-1)(D_P.\widehat{\omega}) -  \widehat{\omega}^2 + \Psi_P^2
$$
and similarly for $Q$. Therefore
$$
\begin{aligned}
\hat{h}(j(P))&+\hat{h}(j(Q))- 2g\langle j(P),j(Q)\rangle \\
&= 8g(g-1)^2(D_P.D_Q) + 2(g-1)\widehat{\omega}^2+\Psi_P^2+\Psi_Q^2- 2g(\Psi_P.\Psi_Q).
\end{aligned}
$$
Here we note that $\Psi_P^2 = \sum_{\pfrak\in S}\Psi_{\pfrak,P}^2$ has absolute value bounded by $\sum_{\pfrak}\phi_\pfrak\log \Norm\pfrak$ and similarly for $Q$ (see Theorem \ref{Thmphip}). Furthermore,
$$
|(\Psi_P.\Psi_Q)|\le \sqrt{|\Psi_P^2|\cdot |\Psi_Q^2|}\le \sum_{\pfrak}\phi_\pfrak\log \Norm\pfrak
$$
by Cauchy--Schwarz, as the Arakelov intersection is negative semidefinite on fibral divisors.

Let us now impose the assumption $P\ne Q$ again. By Theorem \ref{ThmSelf}, the estimate for $(D_P.D_Q)$ in Section \ref{SecSpecial}, and the definition of $\phi_\pfrak(X)$ we get
$$
\begin{aligned}
\hat{h}(j(P))&+\hat{h}(j(Q))- 2g\langle j(P),j(Q)\rangle \\
&\ge -8g(g-1)^2\sum_{v:K\to \C} \gamma_v(v(P),v(Q)) -2(g+1)\sum_\pfrak\phi_\pfrak(X)\log \Norm\pfrak.
\end{aligned}
$$

Using Theorem \ref{ThmBdGreen} we get
$$
\begin{aligned}
8g(g-1)^2&\sum_{v:K\to \C}  \gamma_v(v(P),v(Q))\\
& < \frac{(g-1)^2}{3}\max\{6,g+1\}\sum_{v:K\to \C}\delta(X_v) + 2[K:\Q] g (g-1)^2\left(  3g\log g + 16  \right)
\end{aligned}
$$
which finishes the proof.
\end{proof}

\subsection{The main result} Let $X$ be as in Theorem \ref{ThmMain}, in particular, $g\ge 2$.

The next result is Equation (5) in \cite{Fujimori} (note that their function $f$ is our function $j:X\to J$):
\begin{lemma}\label{LemmaAutHt} If $P\in X(K)$ and $\sigma\in G$, then $\hat{h}(j(P))=\hat{h}(j(\sigma(P)))$.
\end{lemma}

\begin{lemma} \label{LemmaPosM} We have $M(X)>0$.
\end{lemma}
\begin{proof} From the definition we know $\phi_\pfrak\ge 0$, and by Theorem \ref{ThmWilms} we have 
$$
\sum_{v:K\to \C} \delta(X_v)\ge -2[K:\Q]g\log(2\pi^4).
$$
Hence, $M(X)/(2g(g-1)^2[K:\Q])$ is bigger than or equal to
$$
\frac{-1}{3}\max\{6,g+1\}\log(2\pi^4) + 3g\log g+16>9
$$
because $g\ge 2$.
\end{proof}

Let us now prove our bound for the height of rational points.

\begin{proof}[Proof of Theorem \ref{ThmMain}] Let $P\in X(K)$ be a rational point with trivial $H$-stabilizer. Let $\Omega \subseteq X(K)$ be its orbit under $H$; thus, $\#\Omega=n$ where $n=\# H$. 

Consider the euclidean space $V=J(K)\otimes \R$ of dimension $r=\rk J(K)$ under the N\'eron--Tate inner product $\langle-,-\rangle$. The image of a point $x\in J(K)$ in $V$ will be denoted by $x'\in V$.

Let $h=\hat{h}(j(P))$ and note that for all $Q\in \Omega$ we have $\hat{h}(j(Q))=h$, by Lemma \ref{LemmaAutHt}. As we want to give an upper bound for $h$ and the quantity $M(X)/(2\tau)$ is positive (cf. Lemma \ref{LemmaPosM}), it suffices to consider the case $h>0$. Thus, we also assume $r\ge 1$. 

First we consider the case $r=1$. As $n\ge 3$, there are two distinct points $Q,R\in \Omega$ with $j(Q)'=j(R)'$ in $V\simeq \R$ (they have the same norm $h$). By Theorem \ref{ThmGap} we have $2h - 2gh\ge -M(X)$,  hence 
$$
h\le \frac{M(X)}{2(g-1)} = \frac{M(X)}{2g\tau}.
$$
Let us now consider the case $r\ge 2$. From the definition of $\theta(r,n)$, there are two distinct points $Q,R\in \Omega$ whose angle $\alpha$ in $V$ satisfies
$$
\cos\alpha \ge \cos\theta(r,n) 
$$
even if the image of $\Omega$ in $V$ has less than $n$ elements (in that case, $\alpha=0$). By Theorem \ref{ThmGap} we deduce
$$
M(X)\ge 2gh\cos\alpha - 2h \ge 2h\cdot (g\cos\theta(r,n) -1) = 2\tau g h.
$$
Since $\tau>0$, we get the result.
\end{proof}


\section{Computations}\label{SecComputations}

\subsection{Genus $2$ curves over $\Q$}\label{SecComputations1}

To illustrate the applicability of Theorem \ref{ThmMain}, in this section we will work out an example in the simplest non-trivial case: genus $2$ hyperelliptic curves over $\Q$. In this case, the constant $M(X)$ takes the form
$$
\begin{aligned}
M(X)&=2\delta(X_\C)+6\sum_p\phi_p\log p + 8(3\log 2 + 8)\\
&< 2\delta(X_\C)+6\sum_p\phi_p\log p +80.64.
\end{aligned}
$$
where $X_\C$ is the Riemann surface of the complex points of $X$. Furthermore, if one wants to use Theorem \ref{ThmDeltaH} to bound $\delta(X_\C)$ in terms of the Faltings height $h(J):=h_\Q(J)$ of the jacobian $J$ of $X$, one obtains
\begin{equation}\label{EqnBdMH}
\begin{aligned}
M(X) &< 24 h(J)+16\log (2\pi) + 6\sum_p\phi_p\log p +80.64\\
&< 24 h(J)+6\sum_p\phi_p\log p + 110.05.
\end{aligned}
\end{equation}

We will consider $X$ as given by a hyperelliptic equation
$$
y^2 = a_6x^6+a_5x^5+a_4x^4+a_3x^3+a_2x^2+a_1x+a_0
$$
with integral coefficients. The \emph{height constant} studied in \cite{StollHtCt}, which we denote by $c_X$ (although it depends on the chosen equation), satisfies that for every $P\in X(K)$ 
$$
2h(x(P)) \le \hat{h}(j(P)) +c_X.
$$
where $x(P)$ is the $x$-coordinate of $P$. In the notation of \cite{StollHtCt} this can be seen because the first $3$ coordinates of the Kummer map are given by $1,2x(P), x(P)^2$, (see Chapter 2 in \cite{CasselsFlynn}) and in particular one recovers the point $[1:x(P)^2]\in \Pro^1(\Q)$. 

In order to work out our example using Theorem \ref{ThmMain} we need:
\begin{itemize}
\item The automorphism group $G$ of $X$ over $\Q$. This will be obtained from Magma using the command AutomorphismGroup.
\item An upper bound for $\delta(X_\C)$. This can be obtained by direct computation of this invariant, or by bounding $h(J)$. 
\item Information about the bad fibres of $\Xcal/\Spec \Z$. This will be obtained using the Magma command RegularModel.
\item An upper bound for the height constant $c_X$. This will be obtained using the Magma command HeightConstant.
\end{itemize}

\subsection{An example} As promised in the introduction, let us now compute all the rational points of the genus $2$ curve
$$
X:\quad y^2=x^6+x^4+x^2+1.
$$

Here we have that $G=\Aut_\Q(X)\simeq D_4$ has order $8$, while $\rk J(\Q)=2$. From the table in the introduction we get $\tau > 0.20$, which is positive, so Theorem \ref{ThmMain} applies.

We take $H=G$. This group is generated by the involutions
$$
\begin{aligned}
s &: (x,y)\mapsto (x,-y)\\
t &: (x,y)\mapsto (-x,y)\\
u &: (x,y)\mapsto \left(\frac{1}{x},\frac{y}{x^3}\right).
\end{aligned}
$$
On checks that the only points in $X(\Q)$ with non-trivial $G$-stabilizer are (in weighted projective coordinates):
\begin{equation}\label{EqnPts1}
[1:-1:0],[1:1:0],[-1:-2:1],[-1:2:1],[1:-2:1],[1:2:1],[0:-1:1],[0:1:1].
\end{equation}
We claim that these are all the rational points of $X$. In what follows, we use Magma commands for the computations.

First we need to bound $M(X)$. Using BadPrimes we see that the only bad prime is $p=2$ and we need to compute $\phi_2$. Using RegularModel at $2$  we get that the special fibre at $2$ has $9$ components of multiplicities $[ 1, 1, 1, 1, 1, 1, 2, 2, 2 ]$ and intersection matrix
$$
\Mcal=\left[\begin{array}{rrrrrrrrr}
-2&0&0&0&0&0&1&0&0\\
0&-2&0&0&0&0&1&0&0\\
0&0&-2&0&0&0&1&0&0\\
0&0&0&-2&0&0&1&0&0\\
0&0&0&0&-2&0&0&1&0\\
0&0&0&0&0&-2&0&1&0\\
1&1&1&1&0&0&-3&0&1\\
0&0&0&0&1&1&0&-2&1\\
0&0&0&0&0&0&1&1&-2
\end{array}\right].
$$
The numbers off the diagonal are $0$ and $1$, so the fibre components $C_j$ are geometrically irreducible and they have arithmetic genera $q_j\ge 0$. By Lemma \ref{LemmaSum0} we conclude that $q_j=0$ for each $j$ (one can also check this from the information provided by the command RegularModel). With this we can write the systems $\Xi_k$ for $k=1,...,6$ as these indices are the only cases where $C_k$ has multiplicity $1$. Up to symmetry, the only two relevant systems are $\Xi_1$ and $\Xi_5$, namely
$$
\Mcal \left[\begin{array}{r}
b_1^{(1)}\\
b_2^{(1)}\\
\vdots \\
b_9^{(1)}
\end{array}\right] = \left[\begin{array}{r}
-2\\
0\\
0\\
0\\
0\\
0\\
1\\
0\\
0
\end{array}\right]\quad \mbox{ and }\quad 
\Mcal \left[\begin{array}{r}
b_1^{(5)}\\
b_2^{(5)}\\
\vdots \\
b_9^{(5)}
\end{array}\right] = \left[\begin{array}{r}
0\\
0\\
0\\
0\\
-2\\
0\\
1\\
0\\
0
\end{array}\right]
$$
under the additional conditions $\sum_jb^{(1)}_j=0$ and $\sum_jb^{(5)}_j=0$. The first one gives
$$
\left[\sum_j b^{(1)}_jC_j\, . \, \sum_j b^{(1)}_jC_j\right] = -2
$$
and the second one gives
$$
\left[\sum_j b^{(5)}_jC_j\, . \, \sum_j b^{(5)}_jC_j\right] = -4.
$$
Therefore $\phi_2=4$.

Instead of directly bounding $\delta(X_\C)$, in this example it turns out to be easier to use the Faltings height approach provided by \eqref{EqnBdMH}. For this, the command RichelotIsogenousSurfaces for the jacobian of $X$ shows that $J$ has an isogeny of degree  $4$ to the abelian surface $E\times E$ where $E$ is the elliptic curve
$$
E:\quad y^2=x^3+x^2+x+1.
$$
Using FaltingsHeight we obtain 
$$
h(E) = -0.856...< -0.85. 
$$
Hence
$$
h(J)\le 2h(E)+\frac{\log 4}{2} < -1.
$$
Putting all together, from \eqref{EqnBdMH} we deduce
$$
M(X)<  24\cdot h(J)+ 6\phi_2\log 2 +110.05 <103.
$$
Theorem \ref{ThmMain} finally gives that all $P\in X(\Q)$ having non-trivial stabilizer satisfy
$$
\hat{h}(j(P))\le \frac{M(X)}{2q\tau} < \frac{103}{4\cdot 0.2 } =128.75.
$$
Using the command HeightConstant we obtain that the height constant for $J$ can be taken as $c_X=4.08$. Hence all the points $P$ under consideration satisfy
$$
h(x(P))< \frac{1}{2}(\hat{h}(j(P)) + 4.08)< 67.
$$

This bound is well within the range of search using the Mordell Weil sieve. In fact, in \cite{StollAverage} the Mordell--Weil sieve is used to search for all the rational points of naive multiplicative height up to $10^{1000}$ (hence, with $h(x(P))<2302$) of genus $2$ hyperelliptic curves with integer coefficients of absolute value up to $3$, and Mordell--Weil rank $2$ of the jacobian. The result is that in all these cases the naive multiplicative height of $x(P)$ is at most $209040$, see also \cite{StollCurves}.

The instruction Points(X : Bound:=209040) search for all these points and returns the list that we already had. Therefore, the points in \eqref{EqnPts1} give the complete list of rational points in $X(\Q)$.


\section{Acknowledgments}

N.G.-F. was supported by ANID Fondecyt Regular grant 1211004 from Chile. H.P. was supported by ANID Fondecyt Regular grant 1230507 from Chile. The authors thank J. Balakrishnan, M. Stoll, A. Sutherland, and J. Vonk for comments and useful suggestions on an earlier version of this work.


\end{document}